\documentclass[a4paper,11pt]{article}
\usepackage{amssymb,amsmath,amsfonts,amsthm}
\newtheorem{theorem}{Theorem}[section]

\newtheorem{lemma}[theorem]{Lemma}
\newtheorem{proposition}[theorem]{Proposition}
\newtheorem{definition}[theorem]{Definition}
\newtheorem{remark}{Remark}

\newcommand{\ds}{\displaystyle}

\newcommand{\NN}{\mathbb N}

\newcommand{\CC}{\mathbb C}
\newcommand{\RR}{\mathbb R}
\newcommand{\ZZ}{\mathbb Z}
\newcommand{\EE}{\mathcal E}

\newcommand{\SSS}{\mathcal S}

\allowdisplaybreaks

\newcommand{\beq}{\begin{eqnarray}}
\newcommand{\eeq}{\end{eqnarray}}

\newcommand{\beqs}{\begin{eqnarray*}}
\newcommand{\eeqs}{\end{eqnarray*}}
\begin{document}
\catcode`\@=11

 % definition of section and equation numbering system

  \renewcommand{\theequation}{\thesection.\arabic{equation}}
  \renewcommand{\section}%
  {\setcounter{equation}{0}\@startsection {section}{1}{\z@}{-3.5ex plus -1ex
   minus -.2ex}{2.3ex plus .2ex}{\Large\bf}}
\title{\bf Parametrices and hypoellipticity for pseudodifferential operators on spaces of tempered ultraditributions }
\author{Marco Cappiello $^{\textrm{a}}$, Stevan Pilipovi\'c $^{\textrm{b}}$ and Bojan Prangoski $^{\textrm{c}}$}
\date{}
\maketitle
\noindent
\def\thefootnote{}
\footnote{ 2010 \textit{Mathematical Subject Classification: 47G30, 46F05, 35A17 } \\
\textit{keywords: Tempered ultradistributions, pseudodifferential operators, parametrices, hypoellipticity} }
\begin{abstract}
We construct parametrices for a class of pseudodifferential operators of infinite order acting on spaces of tempered ultradistributions
of Beurling and Roumieu type. As a consequence we obtain a result of hypoellipticity in these spaces.
\end{abstract}

\section{Introduction}
The main concern in this paper is the study of hypoellipticity for pseudodifferential operators in the setting of tempered ultradistributions
of Beurling and Roumieu type on $\RR^d$. These distributions represent the global counterpart of the ultradistributions studied by Komatsu,
see \cite{Komatsu1, Komatsu2, Komatsu3}. We recall that the space of test functions for the ultradistributions of \cite{Komatsu1, Komatsu2, Komatsu3}
is a natural generalisation of the Gevrey classes. In the same way tempered ultradistributions act on a space which generalises the spaces of type
$\mathcal{S}$ introduced by Gelfand and Shilov in \cite{GS}.
\par
Before presenting our results let us recall some previous results on hypoellipticity in the spaces mentioned above. Hypoellipticity in Gevrey classes has been studied by several authors, see \cite{HMM, LR, Rodino, Z} and the references therein. Indeed the functional setting allows to consider very general symbols $a(x,\xi)$ admitting exponential growth at infinity with respect to the covariable $\xi$. This was first noticed in \cite{Z} and generalised in \cite{CM, CZ} with applications to hyperbolic equations in Gevrey classes. In \cite{Z} the hypoellipticity has been obtained
by means of the construction of a parametrix. More recently, the results of \cite{Z} have been extended by Fern\'andez et al. \cite{FGJ} to the space of ultradistributions of Beurling type and by the first author to the global frame of the Gelfand-Shilov spaces of type $\mathcal{S}$, see \cite{C1,C2,C3}, allowing exponential growth for the symbols also with respect to the variable $x$.
\par
It is then natural to study the same problem for pseudodifferential operators acting on tempered ultradistributions. In a recent paper \cite{BojanS}, the third author constructed a global calculus for pseudodifferential operators of infinite order of Shubin type in this setting. Here we want to
apply this tool to construct parametrices for the class of \cite{BojanS} and to prove a hypoellipticity result.
\par
Let us first fix some notation and introduce the functional setting where our results are obtained. In the sequel, the sets of integer, non-negative integer, positive integer, real and complex numbers are denoted by $\ZZ$, $\NN$, $\ZZ_+$, $\RR$, $\CC$. We denote $\langle x\rangle =(1+|x|^2)^{1/2} $ for $x\in \RR^d$, $D^{\alpha}= D_1^{\alpha_1}\ldots D_d^{\alpha_d},\quad D_j^
{\alpha_j}={i^{-1}}\partial^{\alpha_j}/{\partial x}^{\alpha_j}$, $\alpha=(\alpha_1,\alpha_2,\ldots,\alpha_d)\in\NN^d$. Finally, fixed $B>0$ we shall denote by $Q_B^c$ the set of all $(x,\xi)\in \RR^{2d}$ for which we have $\langle x \rangle \geq B$ or $\langle \xi \rangle \geq B.$\\
\indent Following \cite{Komatsu1}, in the sequel we shall consider sequences $M_{p}$ of positive numbers such that $M_0=M_1=1$ and satisfying
all or some of the following conditions:\\
\indent $(M.1)$ $M_{p}^{2} \leq M_{p-1} M_{p+1}, \; \; p \in\ZZ_+$;\\
\indent $(M.2)$ $\ds M_{p} \leq c_0H^{p} \min_{0\leq q\leq p} \{M_{p-q} M_{q}\}$, $p,q\in \NN$, for some $c_0,H\geq1$;\\
\indent $(M.3)$  $\ds\sum^{\infty}_{p=q+1}   \frac{M_{p-1}}{M_{p}}\leq c_0q \frac{M_{q}}{M_{q+1}}$, $q\in \ZZ_+$,\\
\indent $(M.4)$ $\ds \left(\frac{M_p}{p!}\right)^2\leq \frac{M_{p-1}}{(p-1)!}\cdot \frac{M_{p+1}}{(p+1)!}$, for all $p\in\ZZ_+$,\\
In some assertions in the sequel we could replace $(M.3)$ by the weaker assumption \\
\indent $(M.3)'$ $\ds \sum_{p=1}^{\infty}\frac{M_{p-1}}{M_p}<\infty,$ \\
cf. \cite{Komatsu1}. It is important to note that  $(M.4)$ implies $(M.1)$.\\
Note that the Gevrey sequence $M_p=p!^s$, $s>1$, satisfies all of these conditions.\\
\indent For a multi-index $\alpha\in\NN^d$, $M_{\alpha}$ will mean $M_{|\alpha|}$, $|\alpha|=\alpha_1+...+\alpha_d$. Recall that the associated function for the sequence $M_{p}$ is defined by
\beqs
M(\rho)=\sup_{p\in\NN}\log_+   \frac{\rho^{p}}{M_{p}} , \; \; \rho > 0.
\eeqs
The function $M(\rho)$ is non-negative, continuous, monotonically increasing, it vanishes for sufficiently small $\rho>0$ and increases more rapidly than $\ln \rho^p$ when $\rho$ tends to infinity, for any $p\in\NN$ (see \cite{Komatsu1}). \par
For $m >0$ and a sequence $M_p$ satisfying the conditions $(M.1)-(M.3)$, we shall denote by $\mathcal{S}_{\infty}^{M_p,m}(\RR^d)$ the Banach space of all functions $\varphi \in \mathcal{C}^{\infty}(\RR^d)$ such that
\begin{equation} \label{norm}
\|\varphi\|_m:=\sup_{\alpha\in \NN^d}\sup_{x \in \RR^d}\frac{m^{|\alpha|}| D^{\alpha}\varphi(x)| e^{M(m|x|)}}{M_{\alpha }}<\infty,
\end{equation}
endowed with the norm in \eqref{norm} and we denote $\ds\SSS^{(M_p)}(\RR^d)=\lim_{\substack{\longleftarrow\\ m\rightarrow\infty}} \SSS^{M_p,m}_{\infty}(\RR^d)$ and $\ds\SSS^{\{M_p\}}(\RR^d)=\lim_{\substack{\longrightarrow\\ m\rightarrow 0}} \SSS^{M_p,m}_{\infty}(\RR^d)$. In the sequel we shall consider simultaneously the two latter spaces by using the common notation $\mathcal{S}^{\ast}(\RR^d)$.
For each space we will consider a suitable symbol class. Definitions and statements will be formulated first for the $(M_p)$ case and then for the $\{M_p\}$ case, using the notation $\ast$. We shall denote by $\mathcal{S}^{\ast \prime}(\RR^d)$ the strong dual space of $\mathcal{S}^{\ast}(\RR^d)$. We refer to \cite{PilipovicK, PilipovicU, PilipovicT} for the properties of $\mathcal{S}^{\ast}(\RR^d)$ and $\mathcal{S}^{\ast \prime}(\RR^d)$. Here we just recall that the Fourier transformation is an automorphism on $\mathcal{S}^{\ast}(\RR^d)$ and on $\mathcal{S}^{\ast \prime}(\RR^d)$ and that for $M_p=p!^s,\, s>1$, we have $M(\rho) \sim \rho^{1/s}$. In this case $\mathcal{S}^{\ast}(\RR^d)$ coincides respectively with the Gelfand-Shilov spaces $\Sigma_s(\RR^d)$ (resp. $\mathcal{S}_s(\RR^d)$) of all functions  $\varphi \in \mathcal{C}^{\infty}(\RR^d)$ such that
$$\sup_{\alpha, \beta \in \NN^d} h^{-|\alpha|-|\beta|}(\alpha!\beta!)^{-s} \sup_{x \in \RR^d}|x^\beta \partial^\alpha \varphi(x)|<\infty$$
for every $h >0$ (resp. for some $h >0$), cf. \cite{GS, PilipovicU}.
\par
Following \cite{BojanS} we now introduce the class of pseudodifferential operators to which our results apply.
Let $M_p, A_p$ be two sequences of positive numbers. We assume that $M_p$ satisfies  $(M.1), (M.2)$ and $(M.3)$ and that $A_p$  satisfies
 $A_0=A_1=1$, $(M.1), (M.2), (M.3)'$ and $(M.4).$ Moreover we suppose that $A_p \subset M_p$ i.e. there exist $c_0 >0, L>0$ such that $A_p \leq c_0 L^p M_p$ for all $p \in \NN.$ Let $\rho_0=\inf\{\rho\in\RR_+|\, A_p\subset M_p^{\rho}\}$. Obviously $0<\rho_0\leq 1$. Let $\rho\in\RR_+$ be arbitrary but fixed such that $\rho_0\leq \rho\leq 1$ if the infimum can be reached, or otherwise $\rho_0<\rho\leq 1$. For any fixed $h>0,m >0$ we denote by $\Gamma_{A_p, \rho}^{M_p, \infty}(\RR^{2d};h,m)$ the space of all functions $a(x,\xi) \in \mathcal{C}^{\infty}(\RR^{2d})$ such that
\begin{equation} \label{defsymb} \sup_{\alpha, \beta \in \ZZ_+^d} \sup_{(x,\xi) \in \RR^{2d}}\frac{|D_{\xi}^\alpha D_x^\beta a (x,\xi)| \langle (x,\xi)\rangle^{\rho|\alpha+\beta|} e^{-(M(m|x|)+M(m|\xi|))}}{h^{|\alpha+\beta|}A_{\alpha}A_{\beta}}<\infty,
\end{equation}
where $M(\cdot)$ is the associated function for the sequence $M_p$. Then we define
\beqs
\Gamma_{A_p, \rho}^{(M_p), \infty}(\RR^{2d})&=&\lim_{\stackrel{\longrightarrow}{m \rightarrow \infty}}\lim_{\stackrel{\longleftarrow}{h \rightarrow 0}}\Gamma_{A_p, \rho}^{M_p,\infty}(\RR^{2d};h,m);\\
\Gamma_{A_p, \rho}^{\{M_p\}, \infty}(\RR^{2d})&=&\lim_{\stackrel{\longrightarrow}{h \rightarrow \infty}}\lim_{\stackrel{\longleftarrow}{m \rightarrow 0}}\Gamma_{A_p, \rho}^{M_p,\infty}(\RR^{2d};h,m).
\eeqs
\begin{remark}
We notice that in the case $M_p=p!^{s}, s>1,$ we can replace $M(m|x|)+M(m|\xi|)$ by $M(m(|x|+|\xi|))$ in \eqref{defsymb}. In particular, in the case of non-quasi-analytic Gelfand-Shilov spaces, we can include symbols of the form $e^{\pm \langle (x,\xi) \rangle^{1/s}}$ in our class, cf. \cite{PT}. \end{remark}
We  associate to any symbol $a \in \Gamma^{\ast, \infty}_{A_p, \rho}(\RR^{2d})$ a pseudodifferential operator $a(x,D)$ defined, as it is usual, by
\begin{equation}\label{pseudodif}
a(x,D)f(x)= (2\pi)^{-d} \int_{\RR^d} e^{i\langle x,\xi \rangle} a(x,\xi) \hat{f}(\xi) d\xi, \qquad f \in \mathcal{S}^{\ast}(\RR^d),
\end{equation}
where $\hat{f}$ denotes the Fourier transform of $f$. In \cite{BojanS} it was proved that operators of the form \eqref{pseudodif} act continuously on $\mathcal{S}^{\ast}(\RR^d)$ and on $\mathcal{S}^{\ast \prime}(\RR^d)$. Moreover, a symbolic calculus for $\Gamma^{\ast, \infty}_{A_p, \rho}(\RR^{2d})$ (denoted there by $\Gamma^{\ast, \infty}_{A_p, A_p, \rho}(\RR^{2d})$) has been constructed. As a consequence it was proved that the class of pseudodifferential operators with symbols in $\Gamma^{\ast, \infty}_{A_p, \rho}(\RR^{2d})$ is closed with respect to composition and adjoints. Here we introduce a notion of hypoellipticity for this class.

\begin{definition}
Let $a\in\Gamma^{*,\infty}_{A_p,\rho}\left(\RR^{2d}\right)$. We say that $a$ is $\Gamma^{*,\infty}_{A_p,\rho}$-hypoelliptic if
\begin{itemize}
\item[$i$)] there exists $B>0$ such that there exist $c,m>0$ (resp. for every $m>0$ there exists $c>0$) such that
\beq\label{dd1}
|a(x,\xi)|\geq c e^{-M(m|x|)-M(m|\xi|)},\quad (x,\xi)\in Q^c_B
\eeq
\item[$ii$)] there exists $B>0$ such that for every $h>0$ there exists $C>0$ (resp. there exist $h,C>0$) such that
\beq\label{dd2}
\left|D^{\alpha}_{\xi}D^{\beta}_x a(x,\xi)\right|\leq C\frac{h^{|\alpha|+|\beta|}|a(x,\xi)|A_{\alpha}A_{\beta}}{\langle(x,\xi)\rangle^{\rho(|\alpha|+|\beta|)}},\,\, \alpha,\beta\in\NN^d,\, (x,\xi)\in Q^c_B.
\eeq
\end{itemize}
\end{definition}

The main result of the paper is the following

\begin{theorem}\label{hypothm}
Let $a\in\Gamma^{*,\infty}_{A_p,\rho}(\RR^{2d})$ be $\Gamma^{*,\infty}_{A_p,\rho}$-hypoelliptic and let $v\in\SSS^*(\RR^d)$. Then every solution $u\in\SSS^{* \prime}(\RR^d)$ to the equation $a(x,D)u=v$ belongs to $\SSS^*(\RR^d)$.
\end{theorem}

\begin{remark}
Note that the symbols of the form $\langle(x, \xi)\rangle^k$ (for $k$ real) work well as  hypoelliptic symbols in the case of the Gevrey sequence $M_p=p!^s, s>1.$ In the case $M_p=p!^s, s>2,$ symbols of the form $e^{\langle (x,\xi)\rangle^{1/s}}$ satisfy the conditions \eqref{dd1}, \eqref{dd2} (cf. \cite[Section 5]{PT} for other examples of hypoelliptic operators in another context). The more sophistic analysis for $s>1$ will be considered separately in a forthcoming paper.
 
In \cite{GPR} the authors characterize  Gelfand-Shilov spaces through the Fourier expansions of their elements by the eigenfunctions of a positive globally elliptic Shubin type operator, cf. \cite{Sh}, and  the sub-exponential growth with eigenvalues
 of  the corresponding Fourier coefficients. With this, one can  verify that the lower bound assumption \eqref{dd1} is sharp if we consider operators of the form $\exp (-P^{1/ms})u := \sum_{j=1}^\infty e^{-\lambda_j^{1/ms}}u_j \varphi_j,$ where $P$ is a positive globally elliptic Shubin differential operator of order $m$, $\lambda_j$ are its eigenvalues, $\{\varphi_j\}_{j \in \NN}$ is an orthonormal basis of eigenfunctions of $P$ and $u_j$ denote the Fourier coefficients of $u$.
\end{remark}

The proof of Theorem \ref{hypothm} is based on the construction of a parametrix for a $\Gamma^{*,\infty}_{A_p,\rho}$-hypoelliptic operator.
To perform this step we use the global calculus developed in \cite{BojanS}. In Section \ref{sec1} we recall some facts about this calculus. Section \ref{sec2} is devoted to the construction of the parametrix and to the proof of Theorem \ref{hypothm}.

\section{Pseudodifferential operators on $\mathcal{S}^{*}(\RR^d), \mathcal{S}^{*\prime}(\RR^d)$} \label{sec1}
In this section we recall some facts about the pseudodifferential calculus for operators with symbols in $\Gamma^{\ast, \infty}_{A_p, \rho}(\RR^{2d})$ which will be used in the proofs of the next section. Since the statements below are proved in \cite{BojanS} for slightly more general classes of symbols, we prefer to report here the same results as they should be read for the class $\Gamma^{\ast, \infty}_{A_p, \rho}(\RR^{2d})$ in order to make the paper self-contained. For proofs and further details we refer to \cite{BojanS}.
First we recall the notion of asymptotic expansion, cf. \cite[Definition 2]{BojanS}.

\begin{definition}
Let $M_p$ and $A_p$ be as in the definition of $\Gamma_{A_p, \rho}^{*, \infty}(\RR^{2d})$ and let $m_0=0, m_p=M_p/M_{p-1}, p \in \ZZ_+$. We denote by $FS^{\ast, \infty}_{A_p, \rho}(\RR^{2d})$ the space of all formal sums $\sum_{j \in \NN}a_j$
such that for some $B>0$, $a_j \in \mathcal{C}^{\infty}(\textrm{int} \, Q_{Bm_j}^c)$ and satisfy the following condition: there exists $m>0$ such that for every $h >0$ (resp. there exists $h >0$ such that for every $m>0$) we have
$$\sup_{j \in \NN} \sup_{\alpha, \beta \in \NN^d} \sup_{(x,\xi) \in Q^c_{Bm_j}} \frac{|D_{\xi}^{\alpha}D_{x}^{\beta}a_j(x,\xi)|\langle (x,\xi)\rangle^{\rho(|\alpha+\beta|+2j)}e^{-M(m|x|)-M(m|\xi|)}}{h^{|\alpha+\beta|+2j}A_{\alpha}A_\beta A_j^2}<\infty.$$
\end{definition}

Notice that any symbol $a \in \Gamma^{\ast, \infty}_{A_p, \rho}(\RR^{2d})$ can be regarded as an element $\sum\limits_{j \in \NN}a_j$ of $FS^{\ast, \infty}_{A_p, \rho}(\RR^{2d})$ with $a_0=a, a_j =0$ for $j \geq 1.$

\begin{definition}
\label{asexp}
A symbol $a \in \Gamma^{\ast, \infty}_{A_p, \rho}(\RR^{2d})$ is equivalent to $\sum_{j \in \NN}a_j \in FS^{\ast, \infty}_{A_p, \rho}(\RR^{2d})$ (we write $a \sim \sum_{j \in \NN}a_j $ in this case) if there exist $m,B >0$ such that for every $h >0$ (resp. there exist $h,B>0$ such that for every $m >0$) the following condition holds:
$$\sup_{N \in \ZZ_+} \sup_{\alpha, \beta \in \NN^d} \sup_{(x,\xi) \in Q^c_{Bm_N}} \frac{\Big|D_{\xi}^{\alpha}D_{x}^{\beta}\big(a(x,\xi) - \sum\limits_{j<N}a_j(x,\xi)\big)\Big|e^{-M(m|x|)-M(m|\xi|)}}{h^{|\alpha+\beta|+2N}A_{\alpha}A_\beta A_N^2 \langle (x,\xi)\rangle^{-\rho(|\alpha+\beta|+2N)}}<\infty.$$
\end{definition}

In \cite{BojanS} it was proved that if $a \sim 0$, then the operator $a(x,D)$ is $\ast$-regularizing, i.e. it extends to a continuous map from $\mathcal{S}^{\ast \prime}(\RR^d)$ to $\mathcal{S}^{\ast}(\RR^d)$.
Moreover we have the following result, cf. \cite[Theorem 4]{BojanS}.

\begin{proposition}\label{exp}
Let $\sum_{j \in \NN}a_j \in FS^{\ast, \infty}_{A_p, \rho}(\RR^{2d})$. Then there exists a symbol $a \in \Gamma^{\ast, \infty}_{A_p, \rho}(\RR^{2d})$ such that $a \sim \sum_{j \in \NN}a_j.$
\end{proposition}

Finally we recall the following composition theorem, cf. \cite[Corollary 1]{BojanS}.

\begin{theorem}\label{composition}
Let $a,b \in \Gamma^{\ast, \infty}_{A_p, \rho}(\RR^{2d})$ with asymptotic expansions $a \sim \sum\limits_{j \in \NN}a_j$ and
$b \sim \sum\limits_{j \in \NN}b_j$. Then there exists $c \in \Gamma^{\ast, \infty}_{A_p, \rho}(\RR^{2d})$ and a $\ast$-regularizing operator $T$
such that $a(x,D)b(x,D) = c(x,D)+T.$ Moreover $c$ has the following asymptotic expansion
$$c(x,\xi) \sim \sum_{j \in \NN} \sum_{s+k+l=j} \sum_{|\alpha|=l} \frac{1}{\alpha!}\partial_\xi^\alpha a_s(x,\xi) D_x^{\alpha} b_k(x,\xi).$$
\end{theorem}

%%%%%%%%%%%%%%%%%%%%%%%%%%%%%%%%%%%%%%%%%%%%%%%%%%%%%%%%%%%%%%%%%%%%%%%%%%%%%%%%%%%%%%%%%%%%%%%%%%%%%%%%%%%%%%%%%%%%%%%%%%%%%%%%%%%%%%%%%%%%%%%%

\section{Hypoellipticity and parametrix}\label{sec2}
In this section we construct the symbol of a left (and right) parametrix for a $\Gamma^{\ast, \infty}_{A_p, \rho}$-hypoelliptic operator starting from the asymptotic expansion of the symbol and using the symbolic calculus developed in \cite{BojanS}. To do this we need some preliminary results.

\begin{lemma}\label{tec}
Let $M_p$ be a sequence of positive numbers satisfying $(M.4)$ and $M_0=M_1=1$. Then for all $2\leq q\leq p$, $\ds\left(\frac{M_q}{q!}\right)^{1/(q-1)}\leq \left(\frac{M_p}{p!}\right)^{1/(p-1)}$.
\end{lemma}
\begin{proof}
For brevity in notation put $N_p=M_p/p!$. Then $N_0=N_1=1$ and $N_p$ satisfies $(M.1)$. Morever the sequence $N_{p-1}/N_p$ is monotonically decreasing. It is enough to prove that $N_p^{1/(p-1)}\leq N_{p+1}^{1/p}$ for $p\geq 2$, $p\in\NN$. The proof goes by induction. For $p=2$ one easily verifies this. Assume that it holds for some $p\geq 2$. Then we have
\beqs
N_{p+1}^{2p+2}&\leq& N_p^{p+1}N_{p+2}^{p+1}\leq N_p N_{p+1}^{p-1} N_{p+2}^{p+1}=N_{p+2}^{2p} N_p\left(\frac{N_{p+1}}{N_{p+2}}\right)^{p-1}\\
&\leq&N_{p+2}^{2p} N_p\frac{N_{p-1}}{N_{p}}\cdot...\cdot \frac{N_{1}}{N_2}=N_{p+2}^{2p},
\eeqs
from which the desired inequality follows.
\end{proof}

\begin{lemma}\label{zkk}
Let $M_p$ satisfy $(M.4)$ and $M_0=M_1=1$. Then for all $\alpha,\beta\in\NN^d$ such that $\beta\leq \alpha$ and $1\leq |\beta|\leq|\alpha|-1$ the inequality $\ds {\alpha\choose\beta}M_{\alpha-\beta}M_{\beta}\leq |\alpha|M_{|\alpha|-1}$ holds.
\end{lemma}
\begin{proof} We will consider two cases.\\
\indent\underline{Case 1.} $2\leq |\beta|\leq |\alpha|-2$.\\
If we use Lemma \ref{tec} and the inequality $\ds {\kappa\choose\nu}\leq {|\kappa|\choose|\nu|}$ for $\nu\leq \kappa$, $\kappa,\nu\in \NN^d$, we have
\beqs
{\alpha\choose\beta}M_{\alpha-\beta}M_{\beta}&\leq&|\alpha|! \cdot\frac{M_{\alpha-\beta}}{(|\alpha|-|\beta|)!}\cdot \frac{M_{\beta}}{|\beta|!}\\
&\leq&|\alpha|!\cdot\left(\frac{M_{|\alpha|-1}}{(|\alpha|-1)!}\right)^{\frac{|\alpha|-|\beta|-1}{|\alpha|-2}}\cdot \left(\frac{M_{|\alpha|-1}}{(|\alpha|-1)!}\right)^{\frac{|\beta|-1}{|\alpha|-2}}=|\alpha|M_{|\alpha|-1}.
\eeqs
\indent\underline{Case 2.} $|\beta|=1$ or $|\beta|=|\alpha|-1$.\\
Then obviously $\ds{\alpha\choose\beta}M_{\alpha-\beta}M_{\beta}\leq |\alpha|M_{|\alpha|-1}$.
\end{proof}

In the following we assume that $A_p$ satisfies the conditions $(M.1)$, $(M.2)$, $(M.3)'$ and $(M.4)$. Furthermore we suppose that $A_0=A_1=1$.
Because of $(M.3)'$, $A_p/(pA_{p-1})\rightarrow\infty$, when $p\rightarrow\infty$, see \cite{Komatsu1}.
Under these assumptions we can prove the following result.

\begin{lemma}\label{lpc}
Let $a\in\Gamma^{*,\infty}_{A_p,\rho}\left(\RR^{2d}\right)$ be $\Gamma^{*,\infty}_{A_p,\rho}$-hypoelliptic. Then, the function $p_0(x,\xi)=a(x,\xi)^{-1}$ satisfies the following condition: for every $h>0$ there exists $C>0$ (resp. there exist $h,C>0$) such that
\beq\label{nl1}
\left|D^{\alpha}_{\xi}D^{\beta}_x p_0(x,\xi)\right|\leq C\frac{h^{|\alpha|+|\beta|}|p_0(x,\xi)|A_{\alpha+\beta}}{\langle(x,\xi)\rangle^{\rho(|\alpha|+|\beta|)}},\,\, \alpha,\beta\in\NN^d,\, (x,\xi)\in Q^c_B.
\eeq
\end{lemma}
\begin{proof}
We observe preliminary that $(M.1)$ and $(M.2)$ on $A_p$ imply that (\ref{dd2}) is equivalent to saying that there exists $B>0$ such that for every $h>0$ there exists $C>0$ (resp. there exist $h,C>0$) such that
\beq\label{dd3}
\left|D^{\alpha}_{\xi}D^{\beta}_x a(x,\xi)\right|\leq C\frac{h^{|\alpha+\beta|}|a(x,\xi)|A_{\alpha+\beta}}{\langle(x,\xi)\rangle^{\rho(|\alpha+\beta|)}},\,\, \alpha,\beta\in\NN^d,\, (x,\xi)\in Q^c_B.
\eeq
Then, to simplify the notation, we set $w=(x,\xi)$. First we will consider the $(M_p)$ case. Let $h>0$ be arbitrary but fixed and take $h_1>0$ such that $2^{4d+2}h_1\leq h$. Then there exists $C_{h_1}\geq 1$ such that
\beq\label{pnn}
\left|D^{\alpha}_w a(w)\right|\leq C_{h_1}\frac{h_1^{|\alpha|}|a(w)|A_{\alpha}}{\langle w\rangle^{\rho|\alpha|}},\,\, \alpha\in\NN^{2d},\, w\in Q^c_B.
\eeq
Now, there exists $t\in\ZZ_+$ such that $C_{h_1}\leq 2^t$. Then, for $|\alpha|\geq t$,
\beq\label{prr}
\left|D^{\alpha}_w a(w)\right|\leq \frac{(2h_1)^{|\alpha|}|a(w)|A_{\alpha}}{\langle w\rangle^{\rho|\alpha|}},\,\, w\in Q^c_B.
\eeq
Choose $s\in\NN$, $s>t+1$, such that
\beq\label{nrr}
C_{h_1}s'A_{s'-1}\leq A_{s'},\, \mbox{ for all } s'\geq s.
\eeq
We will prove that
\beq\label{zdd}
\left|D^{\alpha}_w p_0(w)\right|\leq C_{h_1}^{\min\{s,|\alpha|\}}\frac{h^{|\alpha|}|p_0(w)|A_{\alpha}}{\langle w\rangle^{\rho|\alpha|}},\,\, \alpha\in\NN^{2d},\, w\in Q^c_B,
\eeq
which will complete the proof in the $(M_p)$ case.\\
\indent For $|\alpha|=0$, (\ref{zdd}) is obviously true. Suppose that it is true for $|\alpha|\leq k$, for some $0\leq k\leq s-1$. We will prove that it holds for $|\alpha|=k+1$. If we differentiate the equality $a(w)p_0(w)=1$ on $Q^c_B$, we have
\beqs
|a(w)||D^{\alpha}_w p_0(w)|&\leq& \sum_{\substack{\beta\leq \alpha\\ \beta\neq 0}} {\alpha\choose\beta}|D^{\alpha-\beta}_w p_0(w)| \cdot |D^{\beta}_w a(w)|.
\eeqs
We can use the inductive hypothesis for the terms $|D^{\alpha-\beta}_w p_0(w)|$, Lemma \ref{zkk} and the fact that $qA_{q-1}\leq A_q$, $\forall q\in\ZZ_+$, (which follows from $(M.4)$) to obtain
\beqs
\left|D^{\alpha}_w p_0(w)\right|&\leq& \frac{C_{h_1}^{k+1}|p_0(w)|}{\langle w\rangle^{\rho|\alpha|}}\sum_{\substack{\beta\leq \alpha\\ \beta\neq 0}} {\alpha\choose\beta}h^{|\alpha|-|\beta|} h_1^{|\beta|} A_{\alpha-\beta}A_{\beta}\\
&\leq&\frac{C_{h_1}^{k+1}|p_0(w)|h^{|\alpha|}A_{\alpha}}{\langle w\rangle^{\rho|\alpha|}}\sum_{\substack{\beta\leq \alpha\\ \beta\neq 0}}\left(\frac{h_1}{h}\right)^{|\beta|}\\
&\leq&\frac{C_{h_1}^{k+1}|p_0(w)|h^{|\alpha|}A_{\alpha}}{\langle w\rangle^{\rho|\alpha|}}\sum_{r=1}^{\infty}\left(\frac{h_1}{h}\right)^r\sum_{|\beta|=r}1.
\eeqs
Since
\beqs
\sum_{r=1}^{\infty}\left(\frac{h_1}{h}\right)^r\sum_{|\beta|=r}1\leq \sum_{r=1}^{\infty}{{r+2d-1}\choose{2d-1}}\left(\frac{h_1}{h}\right)^r\leq \sum_{r=1}^{\infty}\left(\frac{2^{4d}h_1}{h}\right)^r\leq 1,
\eeqs
(\ref{zdd}) is true for $0\leq |\alpha|\leq s$. To continue the induction, assume that it is true for $|\alpha|\leq k$, with $k\geq s$. To prove it for $|\alpha|=k+1$, differentiate the equality $a(w)p_0(w)=1$ for $w \in Q^c_B$. We obtain
\beqs
|a(w)|\left|D^{\alpha}_w p_0(w)\right|\leq \sum_{\substack{\beta\leq \alpha\\ \beta\neq 0,\, \beta\neq \alpha}} {\alpha\choose\beta}\left|D^{\alpha-\beta}_w p_0(w)\right| \left|D^{\beta}_w a(w)\right|+\left|p_0(w)\right|\left|D^{\alpha}_w a(w)\right|.
\eeqs
We can use the inductive hypothesis for the terms $\left|D^{\alpha-\beta}_w p_0(w)\right|$, Lemma \ref{zkk} and (\ref{nrr}) to obtain
\beqs
\left|D^{\alpha}_w p_0(w)\right|&\leq& \frac{C_{h_1}^{s}|p_0(w)|}{\langle w\rangle^{\rho|\alpha|}} \left((2h_1)^{|\alpha|}A_{\alpha}+\sum_{\substack{\beta\leq \alpha\\ \beta\neq 0,\, \beta\neq \alpha}}{\alpha\choose\beta} C_{h_1}h^{|\alpha|-|\beta|} h_1^{|\beta|} A_{\alpha-\beta}A_{\beta}\right)\\
&\leq&\frac{C_{h_1}^{s}|p_0(w)|}{\langle w\rangle^{\rho|\alpha|}} \left((2h_1)^{|\alpha|}A_{\alpha}+\sum_{\substack{\beta\leq \alpha\\ \beta\neq 0,\, \beta\neq \alpha}} h^{|\alpha|-|\beta|} h_1^{|\beta|} C_{h_1}|\alpha|A_{|\alpha|-1}\right)\\
&\leq&\frac{C_{h_1}^{s}|p_0(w)|}{\langle w\rangle^{\rho|\alpha|}}\left((2h_1)^{|\alpha|}A_{\alpha}+A_{\alpha}h^{|\alpha|}\sum_{\substack{\beta\leq\alpha\\ \beta\neq 0,\, \beta\neq \alpha}}\left(\frac{h_1}{h}\right)^{|\beta|}\right)\\
&\leq&\frac{C_{h_1}^{s}h^{|\alpha|}|p_0(w)|A_{\alpha}}{\langle w\rangle^{\rho|\alpha|}} \sum_{r=1}^{\infty}\left(\frac{2h_1}{h}\right)^{r}\sum_{|\beta|=r}1\\
&=&\frac{C_{h_1}^{s}h^{|\alpha|}|p_0(w)|A_{\alpha}}{\langle w\rangle^{\rho|\alpha|}} \sum_{r=1}^{\infty}{{r+2d-1}\choose{2d-1}}\left(\frac{2h_1}{h}\right)^{r}.
\eeqs
Finally, we observe that
\beqs
\sum_{r=1}^{\infty}{{r+2d-1}\choose{2d-1}}\left(\frac{2h_1}{h}\right)^{r}\leq \sum_{r=1}^{\infty}\left(\frac{2^{4d+1}h_1}{h}\right)^{r}\leq 1.
\eeqs
This completes the induction.\\
\indent In the $\{M_p\}$ case, there exist $h_1, C_{h_1}>0$ such that (\ref{pnn}) holds. Take $h$ such that $2^{4d+2}h_1\leq h$. Choose $t$ and $s$ as in (\ref{prr}) and (\ref{nrr}). Then we can prove (\ref{zdd}) in the same way as for the $(M_p)$ case.
\end{proof}

\begin{remark}
We observe that to prove Lemma \ref{lpc} we can replace the assumption $(M.4)$ on $A_p$ by a weaker asssumption. Namely we can assume that there exists $K>0$ such that $\ds \left(\frac{M_q}{q!}\right)^{1/q}\leq K \left(\frac{M_p}{p!}\right)^{1/p}$, for all $1\leq q\leq p$. In fact, the latter condition is the same adopted to prove that $1/f\in \EE^*(\RR)$ when $f\in\EE^*(\RR)$ and $\inf|f(x)|\neq 0$ (cf. \cite{Bruna} for the Beurling case and \cite{Rudin} for the Roumieu case). The proof in \cite{Bruna}, \cite{Rudin} relies on careful considerations of the coefficients in the Fa\`a di Bruno formula applied to the composition of the mapping $t\mapsto 1/t$ with $a(x,\xi)$. On the contrary $(M.4)$ is needed to prove the next Lemma \ref{rnc}.
\end{remark}

\begin{lemma}\label{rnc}
Let $a\in\Gamma^{*,\infty}_{A_p,\rho}\left(\RR^{2d}\right)$ be $\Gamma^{*,\infty}_{A_p, \rho}$-hypoelliptic. Define $p_0(x,\xi)=a(x,\xi)^{-1}$ and inductively $$\ds p_j(x,\xi)=-p_0(x,\xi)\sum_{0<|\nu|\leq j} \frac{1}{\nu!}\partial^{\nu}_{\xi}p_{j-|\nu|}(x,\xi)D^{\nu}_x a(x,\xi), j\in\ZZ_+.$$ Then, the functions $p_j$ satisfy the following conditions:\\
\indent there exist $B>0$ such that for every $h>0$ there exists $C>0$ (resp. there exist $h,C>0$) such that
\beq\label{nl2}
\left|D^{\alpha}_{\xi}D^{\beta}_x p_j(x,\xi)\right|&\leq& C\frac{h^{|\alpha|+|\beta|+2j}A_{|\alpha|+|\beta|+2j}|p_0(x,\xi)|} {\langle(x,\xi)\rangle^{\rho(|\alpha|+|\beta|+2j)}},
\eeq
for all $\alpha,\beta\in\NN^d$, $(x,\xi)\in Q^c_B$, $j\in\ZZ_+$;\\
\indent there exist $m,B>0$ such that for every $h>0$ there exists $C>0$ (resp. there exist $h,B>0$ such that for every $m>0$ there exists $C>0$) such that
\beq\label{n12333}
\left|D^{\alpha}_{\xi}D^{\beta}_x p_j(x,\xi)\right|&\leq& C\frac{h^{|\alpha|+|\beta|+2j}A_{|\alpha|+|\beta|+2j}e^{M(m|x|)}e^{M(m|\xi|)}} {\langle(x,\xi)\rangle^{\rho(|\alpha|+|\beta|+2j)}},
\eeq
for all $\alpha,\beta\in\NN^d$, $(x,\xi)\in Q^c_B$, $j\in\ZZ_+$.
\end{lemma}

\begin{proof} First, observe that it is enough to prove (\ref{nl2}) since (\ref{n12333}) follows from (\ref{nl2}) by (\ref{dd1}) (possibly with different constants). As before, we put $w=(x,\xi)$. We will consider first the $(M_p)$ case. Let $h>0$ be fixed. Choose $h_1>0$ so small such that $2^{9d+1}h_1\leq h$ and $e^{4^d dh_1/h}-1\leq 1/2$. Then by assumption and Lemma \ref{lpc}, there exists $C_{h_1}\geq 1$ such that
\beq
\left|D^{\alpha}_w a(w)\right|&\leq& C_{h_1}\frac{h_1^{|\alpha|}|a(w)|A_{\alpha}}{\langle w\rangle^{\rho|\alpha|}},\,\, \alpha\in\NN^{2d},\, w\in Q^c_B,\label{pr1}\\
\left|D^{\alpha}_w p_0(w)\right|&\leq& C_{h_1}\frac{h_1^{|\alpha|}|p_0(w)|A_{\alpha}}{\langle w\rangle^{\rho|\alpha|}},\,\, \alpha\in\NN^{2d},\, w\in Q^c_B,\label{pr2}
\eeq
Take $s\in\ZZ_+$, such that
\beq\label{nrrr}
C_{h_1}^2s'A_{s'-1}\leq A_{s'},\, \mbox{ for all } s'\geq s.
\eeq
We will prove that, for $j \geq 1$,
\beq\label{zdddd}
\left|D^{\alpha}_w p_j(w)\right|\leq C_{h_1}^{2\min\{s,j\}+1}\frac{h^{|\alpha|+2j}A_{|\alpha|+2j}|p_0(w)|} {\langle w\rangle^{\rho(|\alpha|+2j)}},
\eeq
for all $\alpha\in\NN^{2d}$, $w\in Q^c_B$, $j\in\ZZ_+$, which will prove the lemma in the $(M_p)$ case. We can argue by induction on $j$. For $j=1$, we have
\beqs
\left|D^{\alpha}_w p_1(w)\right|&\leq&\sum_{\beta+\gamma+\delta=\alpha}\sum_{|\nu|=1} \frac{\alpha!}{\beta!\gamma!\delta!}\left|D^{\beta}_w p_0(w)\right|\left|D^{\gamma}_wD^{\nu}_{\xi} p_{0}(w)\right|\left|D^{\delta}_w D^{\nu}_x a(w)\right|\\
&\leq&\frac{C_{h_1}^{3}|p_0(w)|}{\langle w\rangle^{\rho(|\alpha|+2)}} \sum_{\beta+\gamma+\delta=\alpha}\frac{d\cdot \alpha!}{\beta!\gamma!\delta!} h_1^{|\beta|}A_{|\beta|}h^{|\gamma|+1}A_{|\gamma|+1} h_1^{|\delta|+1}A_{|\delta|+1}.
\eeqs
For $|\gamma|\geq 1$, by using Lemma \ref{tec}, we obtain
\beqs
A_{|\gamma|+1}\leq (|\gamma|+1)!\left(\frac{A_{|\alpha|+2}}{(|\alpha|+2)!}\right)^{\frac{|\gamma|}{|\alpha|+1}}.
\eeqs
For $|\gamma|=0$ this trivially holds. Also, if $|\beta|\geq 2$,
\beqs
A_{\beta}\leq |\beta|!\left(\frac{A_{|\alpha|+2}}{(|\alpha|+2)!}\right)^{\frac{|\beta|-1}{|\alpha|+1}}\leq |\beta|!\left(\frac{A_{|\alpha|+2}}{(|\alpha|+2)!}\right)^{\frac{|\beta|}{|\alpha|+1}}
\eeqs
and this obviously holds if $|\beta|=1$ or $|\beta|=0$ (note that $(M.4)$ implies that $A_p\geq p!$ for all $p\in\NN$). Moreover for $|\delta|\geq 1$, by Lemma \ref{tec}, we have
\beqs
A_{|\delta|+1}\leq (|\delta|+1)!\left(\frac{A_{|\alpha|+2}}{(|\alpha|+2)!}\right)^{\frac{|\delta|}{|\alpha|+1}}.
\eeqs
If $|\delta|=0$ this inequality obviously holds. Insert these inequalities in the estimate for $\left|D^{\alpha}_w p_1(w)\right|$ to obtain
\beqs
\left|D^{\alpha}_w p_1(w)\right|&\leq&\frac{C_{h_1}^{3}h^{|\alpha|+2}A_{|\alpha|+2}|p_0(w)|}{\langle w\rangle^{\rho(|\alpha|+2)}}\sum_{\beta+\gamma+\delta=\alpha} \frac{d\cdot\alpha!}{\beta!\gamma!\delta!} \left(\frac{h_1}{h}\right)^{|\beta|+|\delta|+1}\\
&{}&\cdot\frac{(|\gamma|+1)! |\beta|!(|\delta|+1)!}{(|\alpha|+2)!}.
\eeqs
Observe that
\beqs
\frac{\alpha!}{\beta!\gamma!\delta!}&=&{\alpha\choose{\beta+\gamma}}{{\beta+\gamma}\choose\beta} \leq{|\alpha|\choose{|\beta+\gamma|}}{{|\beta+\gamma|}\choose|\beta|}\\
&=&\frac{|\alpha|!}{|\beta|!|\gamma|!|\delta|!}\leq \frac{(|\alpha|+1)!}{|\beta|!(|\gamma|+1)!|\delta|!}\leq\frac{(|\alpha|+2)!}{|\beta|!(|\gamma|+1)!(|\delta|+1)!}.
\eeqs
We obtain
\beqs
\left|D^{\alpha}_w p_1(w)\right|\leq\frac{C_{h_1}^{3}h^{|\alpha|+2}A_{|\alpha|+2}|p_0(w)|}{\langle w\rangle^{\rho(|\alpha|+2)}}\sum_{\beta+\gamma+\delta=\alpha} \left(\frac{2^dh_1}{h}\right)^{|\beta|+|\delta|+1}.
\eeqs
Note that
\beqs
\sum_{\beta+\gamma+\delta=\alpha} \left(\frac{2^dh_1}{h}\right)^{|\beta|+|\delta|+1}&\leq& \sum_{l=0}^{\infty}\sum_{|\beta|+|\delta|=l}\left(\frac{2^dh_1}{h}\right)^{l+1}\\
&\leq& \sum_{l=0}^{\infty}{{l+4d-1}\choose{4d-1}}\left(\frac{2^dh_1}{h}\right)^{l+1}\\
&\leq& \sum_{l=0}^{\infty}\left(\frac{2^{9d}h_1}{h}\right)^{l+1} \leq 1,
\eeqs
which completes the proof for $j=1$. Suppose that it holds for all $j\leq k$, $k\leq s-1$, $k\in\ZZ_+$. We will prove it for $j=k+1$.\\
\begin{multline*}
|D^{\alpha}_w p_j(w)| \leq \sum_{\beta+\gamma+\delta=\alpha}\sum_{0<|\nu|\leq j} \frac{\alpha!}{\beta!\gamma!\delta!}\cdot\frac{1}{\nu!}|D^{\beta}_w p_0(w)|\cdot|D^{\gamma}_wD^{\nu}_{\xi} p_{j-|\nu|}(w)| \cdot|D^{\delta}_w D^{\nu}_x a(w)|\\
\leq \frac{C_{h_1}^{2j+1}|p_0(w)|}{\langle w\rangle^{\rho(|\alpha|+2j)}}\sum_{\beta+\gamma+\delta=\alpha}\sum_{0<|\nu|\leq j} \frac{\alpha!}{\beta!\gamma!\delta!\nu!}
\cdot h_1^{|\beta|}A_{|\beta|}h^{|\gamma|+2j-|\nu|}A_{|\gamma|+2j-|\nu|} h_1^{|\delta|+|\nu|}A_{|\delta|+|\nu|},
\end{multline*}
where we used the inductive hypothesis for the derivatives of the terms $p_{j-|\nu|}(w)$. By using Lemma \ref{tec}, we obtain (note that $2j-|\nu|\geq 2$)
\beqs
A_{|\gamma|+2j-|\nu|}&\leq& (|\gamma|+2j-|\nu|)!\left(\frac{A_{|\alpha|+2j}}{(|\alpha|+2j)!}\right)^{\frac{|\gamma|+2j-|\nu|-1}{|\alpha|+2j-1}}\\
&\leq& (|\gamma|+2j-|\nu|)!\left(\frac{A_{|\alpha|+2j}}{(|\alpha|+2j)!}\right)^{\frac{|\gamma|+2j-|\nu|}{|\alpha|+2j-1}},
\eeqs
where the last inequality follows from $A_p\geq p!$, $p\in \NN$, which in turn follows from $(M.4)$. Also, if $|\beta|\geq 2$,
\beqs
A_{\beta}\leq |\beta|!\left(\frac{A_{|\alpha|+2j}}{(|\alpha|+2j)!}\right)^{\frac{|\beta|-1}{|\alpha|+2j-1}}\leq |\beta|!\left(\frac{A_{|\alpha|+2j}}{(|\alpha|+2j)!}\right)^{\frac{|\beta|}{|\alpha|+2j-1}}
\eeqs
and this obviously holds if $|\beta|=1$ or $|\beta|=0$. Moreover for $|\delta|\geq 1$, by Lemma \ref{tec} (because $|\nu|\geq 1$), we have
\beqs
A_{|\delta|+|\nu|}\leq (|\delta|+|\nu|)!\left(\frac{A_{|\alpha|+2j}}{(|\alpha|+2j)!}\right)^{\frac{|\delta|+|\nu|-1}{|\alpha|+2j-1}}.
\eeqs
If $|\delta|=0$ and $|\nu|\geq 2$ Lemma \ref{tec} implies the same inequality and if $|\delta|=0$ and $|\nu|=1$ this inequality obviously holds. If we insert these inequalities in the estimate for $\left|D^{\alpha}_w p_j(w)\right|$, we obtain\\
\\
$\left|D^{\alpha}_w p_j(w)\right|$
\beqs
&\leq&\frac{C_{h_1}^{2j+1}|p_0(w)|}{\langle w\rangle^{\rho(|\alpha|+2j)}} \sum_{\beta+\gamma+\delta=\alpha}\sum_{0<|\nu|\leq j} \frac{\alpha!}{\beta!\gamma!\delta!\nu!}h_1^{|\beta|}h^{|\gamma|+2j-|\nu|} h_1^{|\delta|+|\nu|}\\
&{}&\cdot(|\gamma|+2j-|\nu|)!\left(\frac{A_{|\alpha|+2j}}{(|\alpha|+2j)!}\right)^{\frac{|\gamma|+2j-|\nu|}{|\alpha|+2j-1}} |\beta|!\left(\frac{A_{|\alpha|+2j}}{(|\alpha|+2j)!}\right)^{\frac{|\beta|}{|\alpha|+2j-1}}\\
&{}&(|\delta|+|\nu|)!\left(\frac{A_{|\alpha|+2j}}{(|\alpha|+2j)!}\right)^{\frac{|\delta|+|\nu|-1}{|\alpha|+2j-1}}\\
&=&\frac{C_{h_1}^{2j+1}h^{|\alpha|+2j}A_{|\alpha|+2j}|p_0(w)|}{\langle w\rangle^{\rho(|\alpha|+2j)}} \sum_{\beta+\gamma+\delta=\alpha}\sum_{0<|\nu|\leq j} \frac{\alpha!}{\beta!\gamma!\delta!\nu!} \left(\frac{h_1}{h}\right)^{|\beta|+|\delta|+|\nu|}\\
&{}&\cdot\frac{(|\gamma|+2j-|\nu|)! |\beta|!(|\delta|+|\nu|)!}{(|\alpha|+2j)!}.
\eeqs
Similarly as above, we have
\beqs
\frac{\alpha!}{\beta!\gamma!\delta!}&\leq&\frac{|\alpha|!}{|\beta|!|\gamma|!|\delta|!}\leq \frac{(|\alpha|+2j-|\nu|)!}{|\beta|!(|\gamma|+2j-|\nu|)!|\delta|!}\\
&\leq&\frac{(|\alpha|+2j)!}{|\beta|!(|\gamma|+2j-|\nu|)!(|\delta|+|\nu|)!}.
\eeqs
We obtain
$$|D^{\alpha}_w p_j(w)| \leq
\frac{C_{h_1}^{2j+1}h^{|\alpha|+2j}A_{|\alpha|+2j}|p_0(w)|}{\langle w\rangle^{\rho(|\alpha|+2j)}} \sum_{\beta+\gamma+\delta=\alpha}\sum_{r=1}^{\infty}\sum_{|\nu|=r} \frac{1}{\nu!}\left(\frac{h_1}{h}\right)^{|\beta|+|\delta|+r}.
$$
We have the estimate\\
\\
$\ds \sum_{\beta+\gamma+\delta=\alpha}\sum_{r=1}^{\infty}\sum_{|\nu|=r} \frac{1}{\nu!}\left(\frac{h_1}{h}\right)^{|\beta|+|\delta|+r}$
\beqs
&\leq& \sum_{\beta+\gamma+\delta=\alpha}\sum_{r=1}^{\infty}{{r+d-1}\choose{d-1}} \frac{d^r}{r!}\left(\frac{h_1}{h}\right)^{|\beta|+|\delta|+r}\\
&\leq&\sum_{\beta+\gamma+\delta=\alpha} \left(\frac{h_1}{h}\right)^{|\beta|+|\delta|} \sum_{r=1}^{\infty} \frac{1}{r!}\left(\frac{2^{2d}dh_1}{h}\right)^{r}\\
&=&\left(e^{4^d dh_1/h}-1\right)\sum_{\beta+\gamma+\delta=\alpha} \left(\frac{h_1}{h}\right)^{|\beta|+|\delta|}=\left(e^{4^d dh_1/h}-1\right)\sum_{\beta+\delta\leq\alpha} \left(\frac{h_1}{h}\right)^{|\beta|+|\delta|}\\
&\leq&\left(e^{4^d dh_1/h}-1\right)\sum_{l=0}^{\infty}\left(\frac{h_1}{h}\right)^{l}\sum_{|\beta|+|\delta|=l} 1\\
&=&\left(e^{4^d dh_1/h}-1\right)\sum_{l=0}^{\infty}\left(\frac{h_1}{h}\right)^{l}{{l+4d-1}\choose{4d-1}}\\
&\leq& \left(e^{4^d dh_1/h}-1\right)\sum_{l=0}^{\infty}\left(\frac{2^{8d}h_1}{h}\right)^{l}\leq 1.
\eeqs
Hence, we proved (\ref{zdddd}) for $1\leq j\leq s$. Suppose that it holds for all $j\leq k$, $k\geq s$. For $j=k+1$, similarly as above, we obtain
\beqs
\left|D^{\alpha}_w p_j(w)\right|&\leq& \frac{C_{h_1}^{2s+1}|p_0(w)|}{\langle w\rangle^{\rho(|\alpha|+2j)}}\sum_{\beta+\gamma+\delta=\alpha}\sum_{0<|\nu|\leq j} \frac{\alpha!}{\beta!\gamma!\delta!\nu!}\\
&{}&\cdot C_{h_1}^2h_1^{|\beta|}A_{|\beta|}h^{|\gamma|+2j-|\nu|}A_{|\gamma|+2j-|\nu|} h_1^{|\delta|+|\nu|}A_{|\delta|+|\nu|}.
\eeqs
Note that $|\gamma|+2j-|\nu|\geq s$, so, by (\ref{nrrr}), we have
\beqs
C_{h_1}^2A_{|\gamma|+2j-|\nu|}\leq A_{|\gamma|+2j-|\nu|+1}/(|\gamma|+2j-|\nu|+1).
\eeqs
Also $|\gamma|+2j-|\nu|+1\leq |\alpha|+2j$, hence Lemma \ref{tec} implies
\beqs
C_{h_1}^2A_{|\gamma|+2j-|\nu|}\leq \frac{A_{|\gamma|+2j-|\nu|+1}}{|\gamma|+2j-|\nu|+1}\leq (|\gamma|+2j-|\nu|)!\left(\frac{A_{|\alpha|+2j}}{(|\alpha|+2j)!}\right)^{\frac{|\gamma|+2j-|\nu|}{|\alpha|+2j-1}}.
\eeqs
In the same manner as above we obtain
\beqs
A_{\beta}\leq |\beta|!\left(\frac{A_{|\alpha|+2j}}{(|\alpha|+2j)!}\right)^{\frac{|\beta|}{|\alpha|+2j-1}} \mbox{ and } A_{|\delta|+|\nu|}\leq (|\delta|+|\nu|)!\left(\frac{A_{|\alpha|+2j}}{(|\alpha|+2j)!}\right)^{\frac{|\delta|+|\nu|-1}{|\alpha|+2j-1}}.
\eeqs
If we insert these inequalities in the estimate for $\left|D^{\alpha}_w p_j(w)\right|$ and use the above inequality for $\ds \frac{\alpha!}{\beta!\gamma!\delta!}$ we obtain
\beqs
\left|D^{\alpha}_w p_j(w)\right|\leq\frac{C_{h_1}^{2s+1}h^{|\alpha|+2j}A_{|\alpha|+2j}|p_0(w)|}{\langle w\rangle^{\rho(|\alpha|+2j)}} \sum_{\beta+\gamma+\delta=\alpha}\sum_{r=1}^{\infty}\sum_{|\nu|=r} \frac{1}{\nu!}\left(\frac{h_1}{h}\right)^{|\beta|+|\delta|+r}.
\eeqs
We already proved that $\ds \sum_{\beta+\gamma+\delta=\alpha}\sum_{r=1}^{\infty}\sum_{|\nu|=r} \frac{1}{\nu!}\left(\frac{h_1}{h}\right)^{|\beta|+|\delta|+r}\leq 1$, hence the proof for the $(M_p)$ case is complete.\\
\indent Next, we consider the $\{M_p\}$ case. By assumption and Lemma \ref{lpc}, there exist $h_1,C_{h_1}\geq 1$ such that (\ref{pr1}) and (\ref{pr2}) hold. Take $h$ so large such that $2^{9d+1}h_1\leq h$ and $e^{4^d dh_1/h}-1\leq 1/2$. There exists $s\in\ZZ_+$ such that $C_{h_1}^2 s'A_{s'-1}\leq A_{s'}$, for all $s'\geq s$. One proves that
\beqs
\left|D^{\alpha}_w p_j(w)\right|\leq C_{h_1}^{2\min\{s,j\}+1}\frac{h^{|\alpha|+2j}A_{|\alpha|+2j}|p_0(w)|} {\langle w\rangle^{\rho(|\alpha|+2j)}},
\eeqs
for all $\alpha\in\NN^{2d}$, $w\in Q^c_B$, $j\in\ZZ_+$, by induction on $j$ in the same manner as for (\ref{zdddd}) in the $(M_p)$ case. This completes the proof in the $\{M_p\}$ case.
\end{proof}

\begin{theorem} \label{parametrix}
Let $a\in\Gamma^{*,\infty}_{A_p,\rho}\left(\RR^{2d}\right)$ be $\Gamma^{*,\infty}_{A_p,\rho}$-hypoelliptic. Then there exist *-regularizing operators $T$ and $T'$ and $b,b'\in\Gamma^{*,\infty}_{A_p,\rho}\left(\RR^{2d}\right)$ such that $b(x,D)a(x,D)=\mathrm{Id}+T$ and $a(x,D)b'(x,D)=\mathrm{Id}+T'$.
\end{theorem}

\begin{proof} Let $p_j$, $j\in\NN$, be as in Lemma \ref{rnc}.
%Define $p_0(x,\xi)=a(x,\xi)^{-1}$ and inductively
%\beqs
%p_j(x,\xi)=-p_0(x,\xi)\sum_{0<|\nu|\leq j} \frac{1}{\nu!}\partial^{\nu}_{\xi}p_{j-|\nu|}(x,\xi)D^{\nu}_x a(x,\xi),\,\, j\in\ZZ_+.
%\eeqs
Then the functions $p_0$ and $p_j$, $j\in\ZZ_+$, satisfy the estimates given in Lemmas \ref{lpc} and \ref{rnc}. Since $A_p$ satisfies $(M.1)$ and $(M.2)$, these estimates are equivalent to the following:\\
\indent there exist $m,B>0$ such that for every $h>0$ there exists $C>0$ (resp. there exist $h,B>0$ such that for every $m>0$ there exists $C>0$) such that
\beq\label{k1113}
\left|D^{\alpha}_{\xi}D^{\beta}_x p_j(x,\xi)\right|\leq C\frac{h^{|\alpha|+|\beta|+2j}A_{\alpha}A_{\beta}A_j^2 e^{M(m|x|)} e^{M(m|\xi|)}} {\langle(x,\xi)\rangle^{\rho(|\alpha|+|\beta|+2j)}},
\eeq
for all $\alpha,\beta\in\NN^d$, $(x,\xi)\in Q^c_B$, $j\in\NN$. One can modify $p_0$ near the boundary of $Q^c_B$ so that it can be extended to $\mathcal{C}^{\infty}$ function on $\RR^{2d}$ and satisfy (\ref{k1113}) on the whole $\RR^{2d}$. Hence, (\ref{k1113}) remains true for all $j\in\ZZ_+$ with larger $B$. We obtain $\sum_{j=0}^{\infty}p_j\in FS^{\infty,*}_{A_p,\rho}\left(\RR^{2d}\right)$. Let $b\sim \sum_j p_j$, $b\in\Gamma^{*,\infty}_{A_p,\rho}\left(\RR^{2d}\right)$. By Theorem \ref{composition} there exist $c\in\Gamma^{*,\infty}_{A_p,\rho}\left(\RR^{2d}\right)$ and a *-regularizing operator $\widetilde{T}_1'$ such that $b(x,D)a(x,D)=c(x,D)+\widetilde{T}$ and $c$ has the asymptotic expansion $c\sim \sum_j c_j$, where
\beqs
c_j(x,\xi)=\sum_{s+l=j}\sum_{|\nu|=l}\frac{1}{\nu!}\partial^{\nu}_{\xi} p_s(x,\xi) D^{\nu}_x a(x,\xi).
\eeqs
One easily verifies that $c_0(x,\xi)=1$ on $Q^c_B$. Also, for $j\in\ZZ_+$,
\beqs
c_j=p_j a+\sum_{l=1}^{j}\sum_{|\nu|=l}\frac{1}{\nu!}\partial^{\nu}_{\xi} p_{j-l}\cdot D^{\nu}_x a=p_j a+\sum_{0<|\nu|\leq j} \frac{1}{\nu!}\partial^{\nu}_{\xi}p_{j-|\nu|} \cdot D^{\nu}_x a=0,
\eeqs
on $Q^c_B$, by the definition of $p_j$. Hence, $b(x,D)a(x,D)=\mathrm{Id}+T$ for some *-regularizing operator $T$. With similar constructions one obtains $b'$ such that $a(x,D)b'(x,D)=\mathrm{Id}+T'$, where $T'$ is a *-regularizing operator.
\end{proof}

\textit{Proof of Theorem \ref{hypothm}.} Let $u \in \mathcal{S}^{\ast \prime}(\RR^d)$ be a solution of $a(x,D)u=v \in \mathcal{S}^{\ast}(\RR^d)$. Then,  
applying  the left parametrix $b(x,D)$ of $a(x,D)$, we obtain
$u= b(x,D)v - Tu$ for some *-regularizing operator $T$. Hence $u \in \mathcal{S}^{\ast}(\RR^d)$. The theorem is proved. \qed

\vspace{1cm}
\noindent
$^{\textrm{a}}$ Dipartimento di Matematica, Universit\`a di Torino, via Carlo Alberto 10, 10123 Torino, Italy  \\
{\small\bf email: marco.cappiello@unito.it}\\

\noindent
$^{\textrm{b}}$ Department of Mathematics and Informatics, Faculty of Sciences, University of Novi Sad, Novi Sad, Serbia \\
{\small\bf email: stevan.pilipovic@gmail.com}\\

\noindent
$^{\textrm{c}}$ Department of Mathematics, Faculty of Mechanical Engineering, University Ss. Cyril and Methodius, Skopje, Macedonia  \\
{\small\bf email: bprangoski@yahoo.com}

\end{document}